\documentclass[11pt, a4paper]{amsart}
\newtheorem{prop}{Proposition}[section]
\newtheorem{lema}[prop]{Lemma}
\newtheorem{teo}[prop]{Theorem}
\newtheorem{corolario}[prop]{Corollary}
\newtheorem{defi}[prop]{Definition}

\newtheorem{remark}[prop]{\sc Remark}
\newtheorem{question}[prop]{\sc Question}
\newtheorem{example}[prop]{\sc Example}
\newtheorem{obs}[prop]{\sc Observation}

\newcommand{\supp}{\mbox{supp}}
\newcommand{\uno}{1\!\!1}

\title[Valuations on Star Bodies ]{Radial continuous valuations on star bodies and star sets}

\author{Pedro Tradacete}

\address{Mathematics Department\\ Universidad Carlos III de Madrid \\  28911 Legan\'es (Madrid). Spain.}
\email{ptradace@math.uc3m.es }

\thanks{Support of Spanish MINECO under grants MTM2012-31286 and MTM2013-40985 is gratefully acknowledged.}

\author{Ignacio Villanueva}

\address{Departamento de An\'alisis Matem\'atico \\
Facultad de Matem\'aticas \\ Universidad Complutense de Madrid \\
Madrid 28040}

\email{ignaciov@mat.ucm.es}

\thanks{Partially supported by grants MTM2014-54240-P, funded by MINECO and QUITEMAD+-CM, Reference: S2013/ICE-2801, funded by Comunidad de Madrid}

\begin{document}

\begin{abstract}
We show that a  radial continuous valuation defined on the $n$-dimensional star bodies extends uniquely to a continuous valuation on the $n$-dimensional bounded star sets. Moreover, we provide an integral representation of every such valuation, in terms of the radial function, which is valid on the dense subset of the simple Borel star sets. We also show that every radial continuous valuation defined on the $n$-dimensional star bodies can be decomposed as a sum $V=V^+-V^-$, where both $V^+$ and $V^-$ are positive radial continuous valuations.
\end{abstract}

\subjclass[2010]{52B45, 52A30}

\keywords{Convex geometry; Star bodies; Valuations}

\maketitle

\section{Introduction}

This note continues the study of valuations on star bodies started in \cite{Vi}. A valuation is a function $V$, defined on a class of sets, with the property that
$$
V(A\cup B)+V(A\cap B)=V(A)+V(B).
$$
As a generalization of the notion of measure, valuations have become a relevant area of study in Convex Geometry. In fact, this notion played a critical role in M. Dehn's solution to Hilbert's third problem, asking whether an elementary definition for volume of polytopes was possible. See, for instance, \cite{Lu1}, \cite{Lu2} and the references there included for a  broad vision of the field.

Valuations on convex bodies belong to the Brunn-Minkowski Theory. This theory has been extended in several important ways, and in particular, to the dual Brunn-Minkowski Theory, where convex bodies, Minkowski addition and Hausdorff metric  are replaced by star bodies, radial addition and  radial metric, respectively. The dual Brunn-Minkowski theory, initiated in \cite{Lut_mv1}, has been broadly developed and successfully applied to several areas, such as integral geometry,  local theory of Banach spaces and geometric tomography (see \cite{DGP}, \cite{Gabook} for these and other applications). In particular, it played a key role in the solution of the Busemann-Petty problem \cite{Ga1}, \cite{Ga2}, \cite{Zh}.

D. A. Klain initiated in \cite{Klain96}, \cite{Klain97} the study of rotationally invariant valuations on a certain class of star sets, namely those whose radial function is $n$-th power integrable.

In \cite{Vi}, the second named author started the study of valuations on star bodies, characterizing positive rotational invariant valuations as those described by certain integral representation. In particular, this representation implies that those valuations can be extended to all bounded star sets.

In this note we continue the study of continuous valuations on star bodies, dropping the assumption of rotational invariance. Our main result states that every such valuation can be extended to a continuous valuation on the bounded star sets, and this extension provides an integral representation of the valuation on the star sets with simple Borel radial function. Since every star body (and actually every star set) can be approximated in the radial metric by star sets with simple Borel radial function, for many applications this representation will be useful.

\smallskip

For the sake of clarity, we split the result in two statements. 

\begin{teo}\label{t:extensiontoborel}
Let $V:\mathcal S_0^n\longrightarrow \mathbb R$ be a radial continuous valuation on the $n$-dimensional star bodies $\mathcal S_0^n$. Then, there exists a unique radial continuous extension of $V$ to a valuation $\overline V:\mathcal S_b^n\longrightarrow\mathbb R$ on the bounded Borel star sets of $\mathbb R^n$.
\end{teo}

\begin{teo}\label{t:integral}
Let $V:\mathcal S_0^n\longrightarrow\mathbb R$ be a radial continuous valuation, and let $\overline{V}$ be its extension mentioned in Theorem \ref{t:extensiontoborel}. Then, there exists a measure $\mu$ defined on the Borel sets of $S^{n-1}$ and  a function $K:\mathbb R^+\times S^{n-1}\rightarrow \mathbb R$ such that, for every star body $L$ whose radial function $\rho_L$ is a simple function,  %in $S(\Sigma_n)^+$
we have
$$
\overline V(L)=\int_{S^{n-1}} K(\rho_L(t),t)d\mu(t).
$$
\end{teo}

To prove these results, we need a Jordan-like decomposition  which will probably find applications elsewhere.   We show that every continuous valuation $V:\mathcal S_0^n\longrightarrow \mathbb R$ on the $n$-dimensional star bodies can be decomposed as the difference of two positive continuous valuations. 
With this structural result at hand, the study of  continuous valuations on star bodies reduces to the simpler case of positive continuous valuations.

\begin{teo}\label{t:jordan}
Let $V:\mathcal S_0^n\longrightarrow \mathbb R$ be a radial continuous valuation on the $n$-dimensional star bodies $\mathcal S_0^n$ such that $V(\{0\})=0$. Then, there exist two radial continuous valuations $V^+, V^-:\mathcal S_0^n\longrightarrow \mathbb R_+$ such that $V^+(\{0\})=V^-(\{0\})=0$ and such that
$$
V=V^+-V^-.
$$
Moreover, if $V$ is rotationally invariant, so are $V^+$ and $V^-$.
\end{teo}

In the next paragraphs we describe the structure of the paper.

In Section \ref{sectionnotation} we describe our notation, framework and some known facts that we will need. Then, in Section \ref{preliminaryresults} we show that continuous valuations are bounded on bounded sets, and prove some preliminary results needed later. In Section \ref{s:jordan} we prove Theorem \ref{t:jordan}. As a simple application we solve a question left open in \cite{Vi}.

Section \ref{controlmeasure} is devoted to the construction of control and representing measures associated to a general valuation. It is based on similar work done in \cite{Vi}. Once we have the representing measures, we can extend $V$ to $\overline{V}$ defined on the simple star sets. 

In Section \ref{extension} we prove our main results, Theorems \ref{t:extensiontoborel} and \ref{t:integral}. This is the most technical part of the paper. The main difficulties follow from the fact that we do not know whether $V$, or $\overline{V}$, are uniformly continuous on bounded sets. So, we have to prove in this particular case that $V$  preserves Cauchy sequences and, therefore, can be extended to the bounded star sets.  

Finally, in Section \ref{s:openquestions} we relate our results with existing previous work (\cite{CF}, \cite{FK}, \cite{FK:69}).  In these papers, uniform continuity on bounded sets is assumed a priori, so, much of our difficulties are not present.

\section{Notation and known facts}\label{sectionnotation}

A set $L\subset \mathbb R^n$ is a {\em star set} if it contains the origin and every line through $0$ that meets $L$ does so in a (possibly degenerate) line segment. Let $\mathcal S^n$  denote the set of the star sets of $\mathbb R^n$.

Given $L\in \mathcal S^n$, we define its {\em radial function} $\rho_L$ by
$$
\rho_L(t)=  \sup \{c\geq 0 \, : \, ct\in L\},
$$
for each $t\in\mathbb R^n$. Clearly, radial functions are completely characterized by their restriction to $S^{n-1}$, the euclidean unit sphere in $\mathbb R^n$, so from now on we consider them defined on $S^{n-1}$.

A star set $L$ is called a {\em star body} if $\rho_L$ is continuous. Conversely, given a positive continuous function $f:S^{n-1}\longrightarrow \mathbb R^+=[0,\infty)$ there exists a star body $L_f$ such that $f$ is the radial function of $L_f$.
We denote by $\mathcal S_0^n$ the set of $n$-dimensional star bodies and we denote by $C(S^{n-1})^+$ the set of positive continuous functions on $S^{n-1}$.

Analogously, a star set $L$ is a {\em bounded Borel star set} if $\rho_L$ is a bounded Borel function. Note that star bodies are always bounded. We denote by $\mathcal S_b^n$ the set of $n$-dimensional bounded Borel star sets, $\Sigma_n$ the $\sigma$-algebra of Borel subsets of $S^{n-1}$, and $B(\Sigma_n)^+$ the set of positive bounded Borel functions on $S^{n-1}$.

Given two sets $K,L\in \mathcal S^n$, we define their {\em radial sum} $K\tilde{+}L$ as the star set whose radial function is $\rho_K+\rho_L$. Note that $K\tilde{+}L\in \mathcal S_0^n$ (respectively, $\mathcal S_b^n$) whenever $K,L\in \mathcal S_0^n$ (respectively, $\mathcal S_b^n$).

The dual analog for the Hausdorff metric of convex bodies is the so called {\em radial metric}, which is defined by
$$
\delta(K,L)=\inf\{\lambda\geq 0 : K\subset L\tilde{+} \lambda B_n, L\subset K\tilde{+} \lambda B_n\},
$$
where $B_n$ denotes the euclidean unit ball of $\mathbb R^n$. It is easy to check that
$$
\delta(K,L)=\|\rho_K-\rho_L\|_\infty.
$$

\smallskip

An application $V:\mathcal S^n\longrightarrow \mathbb R$ is a {\em valuation} if for any $K,L\in\mathcal S^n$,
$$
V(K\cup L)+V(K\cap L)=V(K)+ V(L).
$$
It is clear that a linear combination of valuations is a valuation.

\smallskip

Given two functions $f_1, f_2\in B(S^{n-1})^+$, we denote their maximum and minimum by
$$
(f_1\vee f_2)(t)=\max \{f_1(t), f_2(t)\},
$$
$$
(f_1\wedge f_2)(t)=\min \{f_1(t), f_2(t)\}.
$$

Given $K,L\in\mathcal S_0^n$ (respectively $\mathcal S_b^n$), both $K\cup L$ and $K\cap L$ are in $\mathcal S_0^n$ (respectively $\mathcal S_b^n$), and it is easy to see that
$$
\rho_{K\cup L}=\rho_K\vee \rho_L, \hspace{1cm} \rho_{K\cap L}=\rho_K\wedge \rho_L.
$$

With this notation, a valuation $V:\mathcal S_0^n\rightarrow \mathbb R$ induces a function $\tilde V:C(S^{n-1})^+\rightarrow \mathbb R$ given by
$$
\tilde V(f)=V(L_f),
$$
where $L_f$ is the star body whose radial function satisfies $\rho_{L_f}=f$. If $V$ is continuous, then  $\tilde V$ is continuous with respect to the $\|\cdot\|_\infty$ norm in $C(S^{n-1})^+$ and satisfies
$$
\tilde V(f)+\tilde V(g)=\tilde V(f\vee g)+\tilde V(f\wedge g)
$$
for every $f,g\in C(S^{n-1})^+$. Conversely, every such function $\tilde V$ induces a radial continuous valuation on $\mathcal S_0^n$. Similarly, a valuation $V:\mathcal S_b^n\rightarrow \mathbb R$ induces a function $\tilde V:B(\Sigma_n)^+\rightarrow \mathbb R$ with analogous properties, and vice versa.

Given  $A\subset S^{n-1}$, we denote the closure of $A$ by $\overline{A}$.
Given  a function $f:S^{n-1} \longrightarrow \mathbb R$, we define the support of $f$ by $$supp(f)=\overline{\{t\in S^{n-1} : f(t)\not = 0\}},$$ and for any set $G\subset S^{n-1}$, we will write $f\prec G$ if $\supp(f)\subset G$. Conversely, $G\prec f$ denotes that $f(t)\geq1$ for every $t\in G$. Throughout, given $A\subset S^{n-1}$, $\chi_A:S^{n-1}\longrightarrow \mathbb R$ denotes the characteristic function of $A$, and f$\uno=\chi_{S^{n-1}}$ denotes the function constantly equal to 1.

For completeness, we state now a result of \cite{Vi} which will be needed later in several ocassions.

\begin{lema}\label{split}\cite[Lemmata 3.3 and  3.4]{Vi}
Let $\{G_i: i\in I\}$ be a family of open subsets of $S^{n-1}$. Let $G=\cup_{i\in I} G_i$. Then, for every $i\in I$ there exists a function $\varphi_i: G \longrightarrow [0,1]$ continuous in $G$ satisfying  $\varphi_i\prec G_i$ and such that $\bigvee_{i\in I} \varphi_i=\uno$ in $G$. Moreover,  let  $f\in C(S^{n-1})^+$ satisfy $f\prec G$. Then, for every $i\in I$, the function  $f_i=\varphi_if$ belongs to $ C(S^{n-1})^+$. Also,  $f_i\prec G_i$ and  $\bigvee_{i\in I} f_i=f$. In particular, for every $i\in I$,  $0\leq f_i\leq f$.
\end{lema}

It should be noted that most of the results stated in this work refering to $C(S^{n-1})$ could also be given for $C(K)$ where $K$ is a compact metrizable space with the same proofs. An adaptation of the results to the non-metrizable setting should also be possible by a standard application of uniformities (cf. \cite {E}).

\section{Preliminary results}\label{preliminaryresults}

To prove our main results we will need to control the maximum value of $V$ on certain sets. The first step in this direction is to show that $V$ is {\em bounded on bounded sets}:

We say that a valuation $V:\mathcal S^n_0\longrightarrow \mathbb R$ is  bounded on bounded sets  if for every $\lambda>0$ there exists  a real number $K>0$ such that, for every star body $L\subset \lambda B_n$, $|V(L)|\leq K.$

Equivalently, $V$ is bounded on bounded sets if for every $\lambda>0$ there exists $K>0$ such that for every $f\in C(S^{n-1})^+$ with $\|f\|_\infty\leq \lambda$ we have $\tilde{V}(f)\leq K$.

\begin{lema}\label{l:bbs}
Every radial continuous valuation $V:\mathcal S^n_0\longrightarrow \mathbb R$ is bounded on bounded sets.
\end{lema}

\begin{proof}
We reason by contradiction. If the result is not true, there exists $\lambda>0$ and  a sequence $(f_i)_{i\in \mathbb N}\subset C(S^{n-1})^+$, with $\|f_i\|_\infty\leq \lambda$ for every $i\in \mathbb N$ and such that $|\tilde{V}(f_i)|\rightarrow +\infty$.

Consider the function $$\theta:\mathbb R^+\longrightarrow \mathbb R$$ defined by $$\theta(c)=\tilde{V}(c \uno).$$ The continuity of $\tilde{V}$ implies that $\theta$ is continuous. Therefore, $\theta$ is uniformly continuous on $[0,\lambda]$. In particular, it is bounded on that interval. Therefore, there exists $M>0$ such that, for every $c\in [0,\lambda]$, $$|\tilde{V}(c \uno)|\leq M.$$

We define  inductively two sequences $(a_j)_{j\in \mathbb N}, (b_j)_{j\in \mathbb N}\subset \mathbb R^+$: Define first $a_0=0$, $b_0=\lambda$. Consider $c_0=\frac{a_0+b_0}{2}$.

We note that  $$\tilde{V}(f_i\vee c_0\uno) +\tilde{V}(f_i\wedge c_0\uno)=\tilde{V}(f_i) + \tilde{V}(c_0\uno).$$

Since $|\tilde{V}(c_0\uno)|\leq M$ and $|\tilde{V}(f_i)|\rightarrow +\infty$, we know that there must exist an infinite set $\mathbb M_1\subset \mathbb N$ such that for $i\in \mathbb M_1$ either $|\tilde{V}(f_i\vee c_0\uno)|\rightarrow +\infty$ or $|\tilde{V}(f_i\wedge c_0\uno)|\rightarrow +\infty$ as $i$ grows to $\infty$. In the first case, we set $a_1=c_0$, $b_1=\lambda$ and $f^1_i=f_i\vee c_0\uno$. In the second case, we set $a_1=0$ and $b_1=c_0$ and $f^1_i=f_i\wedge c_0\uno$. Now we define $c_1=\frac{a_1+ b_1}{2}$ and proceed similarly.

Inductively, we construct two sequences $(a_j), (b_j)\subset \mathbb R^+$, a decreasing sequence of infinite subsets $\mathbb M_j\subset \mathbb N$, and sequences $(f^j_i)_{i\in \mathbb M_j}\subset C(S^{n-1})^+$ such that, for every $j\in\mathbb N$,
$$
|a_j-b_j|=\frac{\lambda}{2^j},
$$
and for every $i\in\mathbb M_j$, for every $t\in S^{n-1}$,
$$
a_j\leq f^j_i(t)\leq b_j,
$$
and with the property that
$$
\lim_{ i\rightarrow \infty} |\tilde{V}(f^j_i)|=+\infty.
$$
Passing to a further subsequence we may assume without loss of generality that, for every $i\in \mathbb N$,  $$|\tilde{V}(f_i^i)|\geq i.$$

Call $d=\lim_i a_i$. If we consider now the sequence $(f_i^i)_{i\in \mathbb N}\subset C(S^{n-1})^+$, we have that $$\|f_i^i-d\uno\|_\infty \rightarrow 0$$ but $$|\tilde{V}(f_i^i)|\geq i,$$
in contradiction to the continuity of $\tilde{V}$ at $d\uno$.

\end{proof}

We thank the anonymous referee of \cite{Vi} for suggesting a procedure very similar to this as an alternative reasoning to show a statement in that paper.

\smallskip

In the rest of this note we will repeatedly use the fact that $S^{n-1}$ is a compact metric space. We will write $d$ to denote the euclidean metric in $S^{n-1}$.

We need to recall an additional concept for our next result:

\begin{defi}\label{def:rims}
Given a set $A\subset S^{n-1}$, and $\omega>0$,  the {\em outer parallel band} around $A$ is the set 
$$
A_\omega=\{t\in S^{n-1} : 0<d(t, A)<\omega\}.
$$
\end{defi}

Note that, for every $A\subset S^{n-1}$ and $\omega>0$, $A_\omega$ is an open set.

In our next result we use the fact that $V$ is bounded on bounded sets to  control  $V$ on these bands.

\begin{lema}\label{rims}
Let $V:\mathcal S^n_0\rightarrow \mathbb R$ be a radial continuous valuation. Let $A\subset S^{n-1}$ be any Borel set and $\lambda\in \mathbb R^+$.
$$
\lim_{\omega\rightarrow 0} \sup\{|\tilde{V}(f)|: \, f\prec A_\omega, \, \|f\|_\infty\leq \lambda\}=0.
$$
\end{lema}

\begin{proof}
We reason by contradiction. Suppose the result is not true. Then there exist $A\subset S^{n-1}$, $\lambda \in \mathbb R^+$, $\epsilon>0$, a sequence $(\omega_i)_{i\in \mathbb N}\subset \mathbb R$  and a sequence $(f_i)_{i\in \mathbb N}\subset C(S^{n-1})^+$ such that $\lim_{i\rightarrow \mathbb N} \omega_i=0$ and, for every $i\in \mathbb N$,

\begin{itemize}
\item   $\omega_i>0$
\item  $f_i\prec A_{\omega_i}$
\item $\|f_i\|_\infty\leq \lambda$
\item $|\tilde{V}(f_i)|\geq \epsilon.$

\end{itemize}

Therefore, there exists an infinite subset $I\subset \mathbb N$ such that either  $\tilde{V}(f_i) >\epsilon$ for every $i\in I$ or $\tilde{V}(f_i) <-\epsilon$ for every $i\in I$. So, we assume without loss of generality that $\tilde{V}(f_i) >\epsilon$ for every $i\in I$. The case $\tilde{V}(f_i) <-\epsilon$ is totally analogous.

Consider $f_1$. Using the continuity of $\tilde{V}$ at $f_1$, we get the existence of $\delta>0$ such that for every  $g\in C(S^{n-1})^+$ with $\|f_1-g\|_\infty<\delta$,
$$
|\tilde{V}(f_1)-\tilde{V}(g)|\leq \frac{\epsilon}{2}.
$$

Since $f_1$ is uniformly continuous and $f_1(t)=0$ for every $t\in A\subset S^{n-1}\backslash A_{\omega_1}$, there exists $0<\rho<\omega_1$ such that, for every $t\in S^{n-1}$ with $d(t, A)<\rho$, $f_1(t)<\delta$.
We consider the disjoint closed sets
$$
C_1=\{t\in S^{n-1} : d(t, A)\leq \frac{\rho}{2}\}
$$
and
$$
C_2=f_1^{-1}\left([\delta, \lambda]\right).
$$
By Urysohn's Lemma, we can consider a continuous function $\psi_1$ with $\psi_{1|_{C_1}}=0$, $\psi_{1|_{C_2}}=1$ and  $0\leq \psi_1(t)\leq 1$ for every $t\in S^{n-1}$. We consider now the function $\psi_1 f_1\in C(S^{n-1})^+$. On the one hand, $\|f_1-\psi_1 f_1\|_\infty\leq \delta$ and, therefore, $$|\tilde{V}(\psi_1 f_1)|\geq \left||\tilde{V}(f_1)|-|\tilde{V}(f_1)-\tilde{V}(\psi_1 f_1)|\right|> \epsilon-\frac{\epsilon}{2}=\frac{\epsilon}{2}.$$

On the other hand,  $\psi_1 f_1\prec A_{\omega_1}\backslash A_{\frac{\rho}{2}}$. Now, we can choose $\omega_{i_2}< \frac{\rho}{2}$ and we can reason similarly as above with the function $f_{i_2}$.

Inductively, we construct a sequence of functions $(\psi_j f_{i_j})_{j\in \mathbb N}\subset C(S^{n-1})^+$ with disjoint support such that $\tilde{V}(\psi_j f_{i_j})>\frac{\epsilon}{2}$. Noting that
$$
\tilde{V}\bigg(\bigvee_j \psi_j f_{i_j}\bigg)=\sum_j \tilde{V}(\psi_j f_{i_j}),
$$
and that
$$
\bigg\|\bigvee_j \psi_j f_{i_j}\bigg\|_\infty\leq\lambda,
$$
we get a contradiction with the fact that $V$ is bounded on bounded sets.
\end{proof}

\section{Proof of Theorem \ref{t:jordan}}\label{s:jordan}

In this section we prove Theorem \ref{t:jordan} and, as a simple application, we complete the main result of \cite{Vi}.

\begin{proof}[Proof of Theorem \ref{t:jordan}]
Let $V:\mathcal S_0^n\longrightarrow \mathbb R$ be as in the hypothesis and consider the associated $\tilde V:C(S^{n-1})^+\longrightarrow\mathbb R$. For every $f\in C(S^{n-1})^+$, we define
$$
\tilde{V}^+(f)=\sup\{\tilde{V}(g): \, 0\leq g \leq f\},
$$
and we consider the application $V^+:\mathcal S_0^n\longrightarrow \mathbb R$ defined by
$V^+(K)=\tilde{V}^+(\rho_K)$.

Assume for the moment that $V^+$ is a radial continuous valuation. In that case, the result follows easily:

First we note that it follows from $\tilde V(0)=0$ that $V^+(\{0\})=0$ and that, for every $f\in C(S^{n-1})^+$, one has   $\tilde V^+(f)\geq 0$. Therefore,  $V^+(K)\geq 0$ for every $K\in \mathcal S_0^n$.

We define next $V^-=V^+-V$. Clearly, $V^-$ is a radial continuous valuation and $V^-(\{0\})=0$. By the definition of $ V^+$, it follows that, for every $K\in \mathcal S_0^n$, one has $V(K)\leq  V^+(K)$. Thus, $V^-(K)\geq 0$. And clearly we have $$V=V^+-V^-.$$

\smallskip

Therefore, we will finish if we show that
$V^+$ is  a radial continuous valuation. Let us prove it.

First, we see that  it is a valuation. Let $f_1, f_2\in C(S^{n-1})^+$. We have to check that
\begin{equation}\label{igualdad}
\tilde{V}^+(f_1\vee f_2)+ \tilde{V}^+(f_1\wedge f_2) = \tilde{V}^+(f_1)+ \tilde{V}^+(f_2).
\end{equation}

Fix $\epsilon>0$. We choose $0\leq g_1\leq f_1$ such that $\tilde{V}^+(f_1)\leq \tilde{V}(g_1)+\epsilon$, and $0\leq g_2\leq f_2$ such that $\tilde{V}^+(f_2)\leq \tilde{V}(g_2)+\epsilon$.

Then,
\begin{align*}
\tilde{V}^+(f_1)+ \tilde{V}^+(f_2)&\leq \tilde{V}(g_1)+ \tilde{V}(g_2)+ 2\epsilon = \tilde{V}(g_1\vee g_2) + \tilde{V}(g_1\wedge g_2) + 2\epsilon\\
&\leq \tilde{V}^+(f_1\vee f_2) + \tilde{V}^+(f_1\wedge f_2) + 2\epsilon,
\end{align*}
where the last inequality follows from the fact that $0\leq g_1\vee g_2\leq f_1\vee f_2$ and $0\leq g_1\wedge  g_2\leq f_1\wedge f_2$. Since $\epsilon>0$ was arbitrary, this proves one of the inequalities in \eqref{igualdad}.

For the other one, fix again $\epsilon>0$. We choose $0\leq g \leq f_1\vee f_2 $ such that $\tilde{V}^+(f_1\vee f_2)\leq \tilde{V}(g)+\epsilon$, and $0\leq h \leq f_1\wedge f_2 $ such that $\tilde{V}^+(f_1\wedge f_2)\leq \tilde{V}(h)+\epsilon$. Let us consider the sets
$$
A=\{ t\in S^{n-1} : f_1(t)\geq f_2(t)\}
$$
and
$$
B=\{ t\in S^{n-1} : f_1(t)< f_2(t)\}.
$$

Let $\lambda=\|f_1\vee f_2\|_\infty$. According to Lemma \ref{rims}, there exists $\omega_1>0$ such that, for every $f\prec A_{\omega_1}$ with $\|f\|_\infty\leq\lambda$ we have $|\tilde{V}(f)|\leq \epsilon$.

Since $\tilde{V}$ is continuous at $g$, there exists $\delta>0$ such that, $|\tilde{V}(g)-\tilde{V}(g')|<\epsilon$ for every $g'$ such that $\|g-g'\|_\infty<  \delta$. We define $g'=(g-\frac{\delta}{2})\vee 0$. Then, for every $t\in A$, it follows that
$$
g'(t)=\max\big\{g(t)-\frac{\delta}{2},0\big\}\leq g(t)\leq (f_1\vee f_2)(t)=f_1(t).
$$
Now, we can apply the uniform continuity of $g$ and $f_1$  to find $\omega_2$ such that for every $t,s\in S^{n-1}$, if $|t-s|<\omega_2$, then $|f_1(t)-f_1(s)|<\delta/4$ and  $|g(t)-g(s)|<\delta/4$. In particular, this implies that for every $t\in A\cup A_{\omega_2}$, $g'(t)\leq f_1(t)$. On the other hand, it is clear that $g'(t)\leq f_2(t)$ for $t\in B$.

Let $\omega=\min\{\omega_1, \omega_2\}$, and let
$$
J(A, \omega)=A\cup A_\omega
$$
be the open $\omega$-outer parallel of the closed set $A$.
Note that $S^{n-1}=J(A, \omega)\cup B$, where both $J(A, \omega)$ and $B$ are open sets. Moreover, we clearly have $J(A, \omega)\cap B=A_\omega$.

We consider the functions  $\varphi_1\prec J(A,\omega)$, $\varphi_2\prec B$  associated to the decomposition $S^{n-1}=J(A, \omega)\cup B$ by Lemma \ref{split}. Then $\varphi_1\vee \varphi_2=\uno$. Let us define $g'_1=\varphi_1 g'$, $g'_2=\varphi_2 g'$, $h_1=\varphi_1 h$, $h_2=\varphi_2 h$ as in Lemma \ref{split}.

A simple verification yields \begin{itemize}
\item $g'=g'_1\vee g'_2$, $h=h_1\vee h_2$,

\item $g'_1\wedge g'_2\prec A_\omega$, $h_1\wedge h_2\prec A_\omega$,

\item $g'_1\wedge h_2\prec A_\omega$, $h_1\wedge g'_2\prec A_\omega$,

\item $0\leq g'_1\vee h_2 \leq f_1$,

\item $0\leq g'_2\vee h_1\leq  f_2$.

\end{itemize}

Therefore, we get
\begin{align*}
\tilde{V^+}(f_1\vee f_2)&+ \tilde{V}^+(f_1\wedge f_2) \leq  \tilde{V}(g)+ \tilde{V}(h) + 2\epsilon \leq \tilde{V}(g')+ \tilde{V}(h) + 3\epsilon \\
&=\tilde{V}(g'_1) + \tilde{V}(g'_2) - \tilde{V}(g'_1\wedge g'_2) + \tilde{V}(h_1) + \tilde{V}(h_2) - \tilde{V}(h_1\wedge h_2)+3\epsilon \\
&\leq\tilde{V}(g'_1) + \tilde{V}(h_2)+ \tilde{V}(g'_2) +\tilde{V}(h_1)+5\epsilon\\
&=\tilde{V}(g'_1\vee h_2)+ \tilde{V}(g'_1\wedge h_2)+ \tilde{V}(g'_2\vee h_1)+\tilde{V}(g'_2\wedge  h_1)+ 5\epsilon\\
&\leq  \tilde{V}(g'_1\vee h_2)+  \tilde{V}(g'_2\vee h_1) + 7\epsilon\leq \tilde{V^+}(f_1)+ \tilde{V}^+(f_2)+7\epsilon.
\end{align*}
Again, since $\epsilon>0$ was arbitrary, this finishes the proof of \eqref{igualdad}.

\smallskip

Let us see now that $\tilde{V}^+$ is continuous. Let us consider $f_0\in C(S^{n-1})^+$ and take $\epsilon>0$. There exists $g_0\in C(S^{n-1})^+$ with $0\leq g_0 \leq f_0$ such that $\tilde{V}^+(f_0)\leq \tilde{V}(g_0)+\epsilon$.

Since $\tilde{V}$ is continuous at $f_0$ and $g_0$,  there exists $\delta>0$ such that for every $f, g \in C(S^{n-1})^+$ with $\|f_0-f\|_\infty<\delta$ and $\|g_0-g\|<\delta$, we have  $|\tilde{V}(f_0)-\tilde{V}(f)|<\epsilon$ and  $|\tilde{V}(g_0)-\tilde{V}(g)|<\epsilon$.

Let now $f\in C(S^{n-1})^+$ be such that $\|f_0-f\|_\infty<\delta$. Pick $g\in C(S^{n-1})^+$ with $0\leq g \leq f$ such that $\tilde{V}^+(f)\leq \tilde{V}(g) + \epsilon$.

Note that  $\|g_0\wedge f - g_0\|<\delta$ and $\|g\vee f_0 - f_0\|<\delta$. Then, we have
$$
\tilde{V}^+(f)\geq \tilde{V}(g_0\wedge f) \geq \tilde{V}(g_0)-\epsilon\geq \tilde{V}^+(f_0)-2\epsilon,
$$
and
\begin{eqnarray*}
\tilde{V}^+(f)&\leq& \tilde{V}(g) + \epsilon =\tilde{V}(g\wedge f_0) + \tilde{V}(g\vee f_0) -  \tilde{V}(f_0) +\epsilon \\
&\leq& \tilde{V}(g\wedge f_0) + |\tilde{V}(g\vee f_0) -  \tilde{V}(f_0) |+ \epsilon \leq \tilde{V}^+(f_0) + 2\epsilon.
\end{eqnarray*}

Hence, $$|\tilde{V}^+(f_0)-\tilde{V}^+(f)|<2\epsilon$$ and $\tilde{V}^+$ is continuous as claimed.

The last statement follows immediately from the proof. 
\end{proof}

As an application of Theorem \ref{t:jordan}, we can complete the main result of \cite{Vi}. In that paper, {\em positive} rotationally invariant continuous valuations $V$ on the star bodies of $\mathbb R^n$, satisfying that $V(\{0\})=0$, are  characterized by an integral representation as in Corollary  \ref{representacion} below. The question of whether a similar description is valid for the case of real-valued (not necessarily positive or negative) continuous rotationally invariant valuations was left open.

Now, Theorem \ref{t:jordan} immediately  gives a positive answer to this question:

\begin{corolario}\label{representacion}
Let $V:\mathcal S_0^n\longrightarrow \mathbb R$ be a rotationally invariant radial continuous valuation on the $n$-dimensional star bodies $\mathcal S_0^n$. Then, there exists a continuous function $\theta:[0,\infty) \longrightarrow \mathbb R$ such that, for every $K\in \mathcal S_0^n$,
$$V(K)=\int_{S^{n-1}} \theta(\rho_K(t)) dm(t),$$
where $\rho_K$ is the radial function of $K$ and $m$ is the Lebesgue measure on $S^{n-1}$ normalized so that $m(S^{n-1})=1$.

Conversely, let $\theta:\mathbb R^+\longrightarrow \mathbb R$ be a continuous function. Then the application  $V:\mathcal S_0^n \longrightarrow \mathbb R$ given by
$$V(K)=\int_{S^{n-1}} \theta(\rho_K(t)) dm(t)$$ is a radial continuous rotationally invariant valuation.
\end{corolario}

\begin{proof}
Let $V:\mathcal S_0^n\longrightarrow \mathbb R$ be a rotationally invariant radial continuous valuation.  Then, the application defined by $V'(L)=V(L)-V(\{0\})$ is easily seen to be a rotationally invariant radial continuous valuation such that $V'(\{0\})=0$. We   decompose it as $V'=V^+-V^-$ as in Theorem \ref{t:jordan}. 

According to \cite[Theorem 1.1]{Vi}, there exist two continuous functions  $\theta^+, \theta^-:[0,\infty) \longrightarrow \mathbb R$ such that, for every $K\in \mathcal S_0^n$,
$$V'(K)=V^+(K)-V^-(K) =\int_{S^{n-1}} \theta^+(\rho_K(t)) dm(t)- \int_{S^{n-1}} \theta^-(\rho_K(t)) dm(t).$$

We define now $\theta=\theta^+-\theta^-+ V(\{0\})$ and the first part of the result follows.

The converse statement had already been proved in \cite[Theorem 1.1]{Vi} (for that implication, the positivity is not needed).
\end{proof}

\smallskip

\begin{remark}
As in \cite{Vi}, the function $\theta$ in Corollary \ref{representacion} is nothing but $\theta(\lambda)=V(\lambda S^{n-1})$.
\end{remark}

\section[Construction of the measures]{Construction of the control measure and the representing measures}\label{controlmeasure}

As in \cite{Vi}, one of the difficulties  we face in the rest of the paper is the fact that $V$ is not defined on star sets, so that we cannot a priori assign a meaning to $V(\chi_A)$. In order to assign a meaning to it, we  proceed in two steps as in \cite{Vi}: for each $\lambda\geq 0$, first we need a {\em control measure} $\mu_\lambda$ that will allow us next to define the {\em representing measure} $\nu_\lambda$ meant to extend $V$ in the sense that $\nu_\lambda(A)$ is the natural assignation for the (not yet defined) $\overline{V}(\lambda \chi_A)$. Even when this measures $\nu_\lambda$ are defined, it is still not obvious how to extend $V$ to the Borel measurable functions. This will be done in the next section.

For each $\lambda\geq 0$ we  construct a control measure associated to a positive radial continuous valuation $V:\mathcal S_0^n\rightarrow\mathbb R^+$ exactly as it was done in \cite{Vi}, since rotational invariance did not play a role in that construction. We sketch the reasonings here and we refer the reader to \cite{Vi} for a more detailed description.

For every $\lambda\geq 0$ we define the outer measure $\mu_{\lambda}^*$ as follows:
For every open set $G\subset S^{n-1}$ we define $$\mu_\lambda^*(G)=\sup \{\tilde{V}(f): \, f\prec G, \, \|f\|_\infty\leq \lambda\}.$$

Next, for every $A\subset S^{n-1}$, we define $$\mu_\lambda^*(A)=\inf\{ \mu_\lambda^*(G):\,  A\subset G, \, G \mbox{ an open set }\}.$$

It is easy to see that both definitions coincide on open sets. It is not difficult to  see that $\mu_\lambda^*$ is an outer measure \cite[Proposition 3.5]{Vi} and that the Borel sets of $S^{n-1}$ are $\mu_\lambda^*$ measurable \cite[Proposition 3.6]{Vi}, so that $\mu_\lambda$, the restriction of $\mu^*_\lambda$ to the Borel $\sigma$-algebra of $S^{n-1}$, is a measure: for $\lambda\geq 0$ and any Borel set $A\subset S^{n-1}$
\begin{equation}\label{def:mu}
\mu_\lambda(A)=\inf\Big\{ \sup \{\tilde{V}(f): \, f\prec G, \, \|f\|_\infty\leq \lambda\}:\,  A\subset G \mbox{ open}\Big\}.
\end{equation}

In \cite{Vi}, the rotational invariance of $V$ was used to show that $\mu_\lambda$ was finite. Now we do not have rotational invariance, but Lemma \ref{l:bbs} yields that, for every $\lambda$,  $\mu_\lambda$ is finite.

We make explicit the control role of the $\mu_\lambda$'s in the following observation:

\begin{obs}\label{Vpequenha} Let $V$ be a positive radial continuous valuation and let $\mu_\lambda$ be the previously defined measure associated to it. For every $\lambda\geq 0$ and  $\epsilon>0$, if   $G\subset S^{n-1}$ is  an open set such that $\mu_\lambda (G)\leq \epsilon$, and $f\in C(S^{n-1})^+$ is such that $f\prec  G$ and $\|f\|_\infty\leq \lambda$, then
$$ \tilde{V}(f)\leq  \epsilon.$$
\end{obs}

\begin{obs}
For every $\lambda\geq0$, $\mu_\lambda$ is a finite Borel measure on the compact metric space $S^{n-1}$. Hence, by Ulam's Theorem, $\mu_\lambda$ is regular (cf. \cite[Theorem 7.1.4]{Dudley}). That is, for every Borel set $A\subset S^{n-1}$ we have
$$
\mu_{\lambda}(A)=\sup\{\mu_\lambda(K):\,  K\subset A, \, K \mbox{ compact}\}=\inf\{\mu_\lambda(G):\, A\subset G, \, G \mbox{ open}\}.
$$
\end{obs}

As in \cite{Vi}, we will now define, for each $\lambda\geq 0$, a measure $\nu_\lambda$ which we will use to represent $\tilde{V}$. Again, we only sketch the construction here and we refer the reader to \cite{Vi} for further details.

We recall that a {\em content} in $S^{n-1}$ is a non negative, finite, monotone set function defined on the family of all closed subsets of $S^{n-1}$, which is finitely subadditive and finitely additive on disjoint sets \cite[\S 53]{Halmos}. For each $\lambda\geq 0$ we define a content in the following way:

\begin{defi}\label{defcontent}For every closed  set $K\subset S^{n-1}$, we define
$$\zeta_\lambda(K)=\inf\{\tilde{V}(f):  K\prec \frac{f}{\lambda},\,\|f\|_\infty\leq\lambda\}.$$
\end{defi}

Thus defined,  $\zeta_\lambda$  can be well approximated from above by decreasing open sets:

\begin{lema}\label{lambdaG}
Let $K\subset G\subset S^{n-1}$ be such that $K$ is closed and $G$ is open.  Then
$$
\zeta_\lambda(K)=\inf\{\tilde{V}(f): K\prec \frac{f}{\lambda} \prec G,\,\|f\|_\infty\leq\lambda\}.
$$
\end{lema}
\begin{proof}
One of the inequalities is trivial. We only need to check that $\zeta_\lambda(K)\geq \inf\{\tilde{V}(f): K\prec \frac{f}{\lambda} \prec G,\,\|f\|_\infty\leq\lambda\}$. To see this, we choose $\epsilon>0$. We pick now $f\in C(S^{n-1})^+$ with $K\prec \frac{f}{\lambda}$ and $\|f\|_\infty\leq\lambda$ such that $\zeta_\lambda(K)\geq \tilde{V}(f)-\epsilon$. The set  $C=\supp(f)\setminus G$ is closed (it could be empty, in that case the next reasonings are trivial). Therefore $C$ is compact, and $K\cap C=\emptyset$.

Since $\mu_\lambda$ is regular, there exists an open set $H\supset C$, with $H\cap K=\emptyset$, such that $\mu_\lambda(H\setminus C)\leq \epsilon$.  Therefore, $\mu_\lambda(G\cap H)\leq \mu_\lambda(H\setminus C)\leq \epsilon$. We apply now Lemma \ref{split} to the open sets $G,H$ and we obtain the functions $\varphi_G, \varphi_H$. We define $f_G=f\varphi_G$ and $f_H=f\varphi_H$. We have that $f=f_G\vee f_H$ and  $\supp(f_G\wedge f_H)\subset G\cap H$. Therefore, Observation \ref{Vpequenha} tells us that  $\tilde{V}(f_G\wedge f_H)\leq \epsilon$. So, we have

\begin{align*}
\zeta_\lambda(K)&\geq \tilde{V}(f)-\epsilon =\tilde{V}(f_G\vee f_H)-\epsilon\geq \tilde{V}(f_G\vee f_H) -\epsilon+ \tilde{V}(f_G\wedge f_H)-\epsilon\\
&=\tilde{V}(f_G)+\tilde{V}(f_H) -2\epsilon \geq \tilde{V}(f_G)-2\epsilon \geq \zeta_\lambda(K)-2\epsilon,
\end{align*}
due to the positivity of $\tilde{V}$.
Since $K\prec \frac{f_G}{\lambda} \prec G$ and $\|f_G\|_\infty\leq \lambda$, our result follows.
\end{proof}

Now, the  fact that $\zeta_\lambda$ is a content, and indeed a regular content,  can be seen exactly as in \cite[Lemmas 4.2 and 4.3]{Vi}. Therefore, we can define a regular measure $\nu_\lambda$ associated to $\zeta_\lambda$ in a standard way (see \cite[\S 53]{Halmos}) by setting, for each Borel set $A\subset S^{n-1}$,
\begin{equation}\label{def:nu}
\nu_\lambda(A)=\inf\{\sup\{\zeta_\lambda(K): K\subset G\}:\,G\textrm{ open},\,A\subset G\}.
\end{equation}
It is easy to see  that, for every closed set $K\subset S^{n-1}$, $\zeta_\lambda(K)=\nu_\lambda(K)$.

These measures $\nu_\lambda$ here defined immediately provide the extension of the valuation to simple Borel star sets. The next two lemmas will allow us to have good control of these measures $\nu_\lambda$.

\begin{lema}\label{l:controlenrims}
Let $C\subset S^{n-1}$ be a closed set, $\lambda\geq 0$, $\epsilon>0$ and let $G\subset S^{n-1}$ be an open set such that $C\subset G$ and such that $\mu_\lambda(G\setminus C)<\epsilon$. Then, for every pair of positive continuous functions $f_1,f_2$ such that, for $j=1,2$, $C\prec \frac{f_j}{\lambda} \prec G$ with $\|f_j\|_\infty\leq\lambda$, we have
$$
|\tilde{V}( f_1)-\tilde{V}( f_2)|\leq 6\epsilon.
$$

\end{lema}

\begin{proof}
Since $\tilde{V}$ is continuous at $ f_1$ and $ f_2$,  there exists $\delta_1>0$ such that, for every $f\in C(S^{n-1})^+$, if $\|f- f_j\|\leq \delta_1$, then $|\tilde{V}(f)-\tilde{V}(  f_j)|<\epsilon$, for $j=1,2$.  We define $\delta=\min\{\delta_1, \frac{\lambda}{2}\}$ (this is just needed to make sure that $\lambda-\delta$ below is strictly greater than 0). Now,  using the fact that  both $ f_1$ and $f_2$ are uniformly continuous, we get the existence of $\rho$ such that, for every $t,s\in S^{n-1}$, $|t-s|<\rho$ implies that $|f_j(t)-  f_j(s)|<\delta$,  $j=1,2$.

Let $$J(C,\rho)=\{t\in S^{n-1}:d(t,C)<\rho\}.$$
The paragraph above implies that, for $j=1,2$, for every $t\in J(C,\rho)$, $ f_j(t)>\lambda-\delta$. For $j=1,2$ we define the functions
$$
\tilde{f}_j=\lambda \uno\wedge \left(\frac{f_j}{1-\frac{\delta}{\lambda} }\right).
$$

We clearly have that $\tilde{f}_j\in C(S^{n-1})^+$, $\tilde{f}_j\prec G$ and, for every  $t\in J(C,\rho)$, $\tilde{f}_j(t)=\lambda$. Also, we have that
$$
\|\tilde{f}_j-f_j\|_\infty \leq  \delta.
$$
(For this last inequality, note that if $f_j(t)\geq \lambda(1-\frac{\delta}{\lambda})$, then $\tilde{f_j}(t)-f_j(t)=\lambda-f_j(t)\leq \delta$. Otherwise, if $f_j(t)<\lambda(1-\frac{\delta}{\lambda})$, we have $\tilde{f}_j(t)-f_j(t)=f_j(t)\left(\frac{1}{1-\frac{\delta}{\lambda}} -1\right)=f_j(t)\frac{\delta}{\lambda-\delta} <\delta$.)

We consider now the open sets $G_1=G\cap \{t\in S^{n-1} : d(t, C)<\frac{2\rho}{3}\}$ and $G_2=G\cap \{t\in S^{n-1} : \frac{\rho}{3}<d(t, C)\}$. We consider two functions $\varphi_i\prec G_i$, $i=1,2$ as in Lemma \ref{split} and for $i=1,2$, $j=1,2$ we define the function $\tilde{f}_j^i=\varphi_i \tilde{f}_j$.

Then, for $j=1,2$, $\tilde{f}_j=\tilde{f}_j^1\vee \tilde{f}_j^2$. Also, for every $t\in G_1$, $\tilde{f}_1(t)=\tilde{f}_2(t)$. Therefore, $\tilde{f}_1^1=\tilde{f}_2^1$. Moreover, for $j=1,2$, $\tilde{f}_j^2\prec G_2\subset G\setminus C$ and, therefore, also $\tilde{f}_j^1\wedge\tilde{f}_j^2\prec G\setminus C$. Hence, by Observation \ref{Vpequenha}, we have that, for $j=1,2$, $\tilde{V}(\tilde{f}_j^2)\leq \epsilon$ and  $\tilde{V}(\tilde{f}_j^1\wedge \tilde{f}_j^2)\leq \epsilon$.

Recalling that,  for $j=1,2$, $$\tilde{V}(\tilde{f}_j)=\tilde{V}(\tilde{f}_j^1)+ \tilde{V}(\tilde{f}_j^2) - \tilde{V}(\tilde{f}_j^1\wedge \tilde{f}_j^2)),$$
we get
\begin{align*}
\left|\tilde{V}(f_1)-\tilde{V}(f_2)\right |&\leq \left|\tilde{V}(f_1)-\tilde{V}(\tilde{f}_1)\right |+\left|\tilde{V}(\tilde{f}_1)-\tilde{V}(\tilde{f}_2)\right |+ \left|\tilde{V}(f_2)-\tilde{V}(\tilde{f}_2)\right |\\
&\leq \left|\tilde{V}(\tilde{f}_1)-\tilde{V}(\tilde{f}_2)\right |+2\epsilon\\
&= \left| \tilde{V}(\tilde{f}_1^2) - \tilde{V}(\tilde{f}_1^1\wedge \tilde{f}_1^2)- \tilde{V}(\tilde{f}_2^2) +\tilde{V}(\tilde{f}_2^1\wedge \tilde{f}_2^2))\right|+2\epsilon\leq 6\epsilon.
\end{align*}
\end{proof}

As an immediate corollary, we have:
\begin{lema}\label{l:aproxnu}
Let $C\subset S^{n-1}$ be a closed set, $\lambda\geq 0$, $\epsilon>0$ and let $G\subset S^{n-1}$ be an open set such that $C\subset G$ and such that $\mu_\lambda(G\setminus C)<\epsilon$. Then,  for every $f\in C(S^{n-1})^+$ such that $\|f\|_\infty\leq \lambda$ and $C\prec \frac{f}{\lambda} \prec G$, $$\nu_\lambda(C)\leq \tilde{V}(f)\leq \nu_\lambda(C)+7\epsilon.$$

Therefore, if for every $\omega>0$ we choose $f_\omega\in C(S^{n-1})^+$ such that $\|f_\omega\|_\infty\leq \lambda$ and $C\prec \frac{f_\omega}{\lambda} \prec C\cup C_\omega$, $$\lim_{\omega\rightarrow 0} \tilde{V}(f_\omega)=\nu_\lambda(C).$$
\end{lema}

\begin{proof}
Using Lemma \ref{lambdaG} and the fact that $\nu_\lambda(C)=\zeta_\lambda(C)$, we can choose $g\in C(S^{n-1})^+$ such that $\|g\|_\infty\leq \lambda$ and $C\prec \frac{g}{\lambda} \prec G$ and $\tilde{V}(g)\leq \nu_\lambda(C)+\epsilon$. Lemma \ref{l:controlenrims} proves now the first part of the statement.

For the second part, it is enough to note that Lemma \ref{rims} implies that  $\mu_\lambda((C\cup C_\omega)\setminus C)=\mu_\lambda(C_\omega)$ tends to 0 as $\omega$ tends to 0.
\end{proof}

\section{Proof of the main result}\label{extension}

In this section we prove Theorems \ref{t:extensiontoborel} and \ref{t:integral}. The main technical difficulty is to prove that $V$, and its extension to the simple functions defined through the measures $\nu_\lambda$, not only are continuous, but they preserve Cauchy sequences.

To do this, we first  show how a positive radial continuous valuation $V:\mathcal S_0^n\longrightarrow \mathbb R^+$ on the star bodies of $\mathbb R^n$ can be extended to a positive radial continuous valuation $\overline{V}:\mathcal S_b^n\longrightarrow \mathbb R^+$ on the {\em bounded Borel star sets} of $\mathbb R^n$. Once this is done, the positivity assumption can be removed using  Theorem \ref{t:jordan}.

As we mentioned in the introduction, the star bodies of $\mathbb R^n$ can be identified, by means of their radial functions, with the cone $C(S^{n-1})^+$ of the positive continuous functions defined on $S^{n-1}$. Similarly,  the star sets of $\mathbb R^n$ can be identified with $B(\Sigma_n)^+$, the positive bounded Borel functions (here, $\Sigma_n$ denotes the $\sigma$-algebra of the Borel subsets of $S^{n-1}$). Let $S(\Sigma_n)^+$ denote the space of positive Borel simple functions. That is, functions of the form  $\sum_{i=1}^n a_i\chi_{A_i}$ with $A_i\in\Sigma_n$ and $a_i\geq 0$ for $i=1,\ldots,n$. Recall that every bounded Borel function is the uniform limit of Borel simple functions.

For simplicity, in the rest of the paper we slightly abuse the notation and, whenever $X$ is one of the spaces $C(S^{n-1})^+$, $B(\Sigma_n)^+$ or $S(\Sigma_n)^+$, we say that an application $V:X\longrightarrow \mathbb R$ is a valuation if, for every $f,g\in X$, $$V(f \vee g)+V(f\wedge g)=V(f)+ V(g).$$

With this notation, the result we need to prove can be stated as: 

\begin{teo}\label{t:extension}
Let $\tilde V:C(S^{n-1})^+\rightarrow \mathbb R^+$ be a  continuous valuation with $\tilde V(0)=0$. Then $\tilde V$ admits a unique continuous extension $\overline V:B(\Sigma_n)^+\rightarrow \mathbb R^+$ which is also a valuation. 
\end{teo}

Also for simplicity, we use the same notation for the valuation $\overline{V}:\mathcal S_b^n\longrightarrow \mathbb R^+$ and its associated function $\overline{V}:B(\Sigma_n)^+\longrightarrow \mathbb R^+$. It will be clear from the context to which of them we refer every time. 

We will start by defining $\overline{V}$ on simple functions: given a simple function $g=\sum_{i=1}^M a_i \chi_{A_i}\in S(\Sigma_n)^+$, with $A_i\cap A_j=\emptyset$ for  $i\not = j$,
we set
$$
\overline{V}(g)=\sum_{i=1}^M \nu_{a_i}(A_i).
$$

We want to extend $\overline{V}$ now to $B(\Sigma_n)^+$, the closure of $S(\Sigma_n)^+$. To do this, we need to show that $\overline{V}:S(\Sigma_n)^+\longrightarrow \mathbb R^+$ preserves  Cauchy sequences. In order to prove this, we need several previous technical results.

The following  lemma is a refinement of Lemma \ref{l:controlenrims}.

\begin{lema}\label{l:controldefunciones}
Let $C\subset S^{n-1}$ be a closed set. Let $\epsilon>0$, $\lambda\geq 0$ and let $G\subset S^{n-1}$ be an open set such that $C\subset G$ and $\mu_{\lambda+1}(G\setminus C)< \epsilon$. For $j=1,2$ let $f_j\in C(S^{n-1})^+$ be such that $f_j\prec G$, $\|f_j\|_\infty\leq \lambda$ and $f_1(t)=f_2(t)$ for every $t\in C$. Then
$$
|\tilde{V}( f_1)-\tilde{V}( f_2)|\leq 8\epsilon.
$$
\end{lema}

\begin{proof}

Since $\tilde{V}$ is continuous at $ f_1$ and $f_2$,  there exists $d_1>0$ such that, for every $f\in C(S^{n-1})^+$, if $\|f- f_j\|_\infty\leq d_1$, then $|\tilde{V}(f)-\tilde{V}( f_j)|<\epsilon$, for $j=1,2$.

We define $d=\min\{d_1,1\}$ and by Urysohn's lemma we consider a function $h\in C(S^{n-1})^+$ such that $0\leq h\leq d$, $h|_{C}=d$ and $h|_{S^{n-1}\backslash G}=0$. For $j=1,2$, let $\hat{f}_j=f_j+h$. We have that, for $j=1,2$, $\hat{f}_j\in C(S^{n-1})^+$ and satisfy

\begin{itemize}
\item $\|\hat{f}_j\|_\infty \leq \lambda+1$

\item $\min_{t\in C} \{\hat{f}_j(t)\}\geq d$

\item $\hat{f}_j\prec G$

\item $|\tilde{V}(\hat{f}_j)-\tilde{V}(f_j)|<\epsilon$.

\end{itemize}

Now we use again the continuity of $\tilde{V}$ to find  $\delta>0$ such that, for $j=1,2$, for every $g\in C(S^{n-1})^+$, if $\|g- \hat{f}_j\|_\infty\leq \delta$ then $|\tilde{V}(g)-\tilde{V}(\hat{f}_j)|<\epsilon$.

We choose a real number  $\alpha$ such that
$$
\frac{1}{1+\frac{\delta}{\lambda+1}}<\alpha<1.
$$

For $j=1,2$ we define the functions $$\tilde{f}_j=(\hat{f}_1\vee \hat{f}_2)\wedge \left(\frac{\hat{f}_j}{\alpha} \right).$$

Then, for $j=1,2$, we have that  $\tilde{f}_j\in C(S^{n-1})^+$ with $\tilde{f}_j\prec G$ and $\|\tilde{f}_j\|_\infty\leq \lambda+1$.

We define $\sigma=\frac{d}{2}\left(\frac{1-\alpha}{1+\alpha}\right)>0$ and,  using the fact that  both $\hat{f}_1$ and $\hat{f}_2$ are uniformly continuous, we get the existence of $\rho$ such that, for $j=1,2$, and for every $t,s\in S^{n-1}$, $|t-s|<\rho$ implies that $|\hat{f}_j(t)-  \hat{f}_j(s)|<\sigma$.

Now, we have that for every $t$ such that $d(t,C)<\rho$,
\begin{equation}\label{eq:igualesenrim}\tilde{f}_1(t)=\tilde{f}_2(t).\end{equation}

Indeed, take $t$ such that $d(t,C)<\rho$. Then, there exists $s_0\in C$ such that $|t-s_0|<\rho$. Therefore, for $j=1,2$, $\hat{f}_j(t)<\hat{f}_j(s_0)+\sigma$ and, hence, $$\hat{f}_1(t)\vee \hat{f}_2(t)<\hat{f}_j(s_0)+\sigma.$$

Moreover $$\frac{\hat{f}_j(t)}{\alpha}>\frac{1}{\alpha}\left(\hat{f}_j(s_0)-\sigma\right).$$

Now, the definition of $\sigma$ implies that $\hat{f}_1(t)\vee \hat{f}_2(t)<\frac{\hat{f}_j(t)}{\alpha}$ and, hence, $\tilde{f}_1(t)=\tilde{f}_2(t)=\hat{f}_1(t)\vee \hat{f}_2(t)$.

\smallskip

Next, we show that $$\|\tilde{f}_j-\hat{f}_j\|_\infty<\delta.$$ To see this, note first that if $\hat{f}_1(t)\vee \hat{f}_2(t)= \hat{f}_j(t)$, then $\tilde{f}_j(t)=\hat{f}_j(t)\wedge \frac{\hat{f}_j(t)}{\alpha}=\hat{f}_j(t)$, so we only need to consider the case when $\hat{f}_1(t)\vee \hat{f}_2(t)= \hat{f}_k(t)$, with $k\not = j$. In that case, $\tilde{f}_j(t)=\hat{f}_k(t)\wedge \frac{\hat{f}_j(t)}{\alpha}$ and we have
\begin{align*}
\left|\tilde{f}_j(t)-\hat{f}_j(t)\right|&=\left|\left(\hat{f}_k(t)\wedge \frac{\hat{f}_j(t)}{\alpha}\right)-\hat{f}_j(t)\right|=\left(\hat{f}_k(t)\wedge \frac{\hat{f}_j(t)}{\alpha}\right)-\hat{f}_j(t)\\
& \leq \frac{\hat{f}_j(t)}{\alpha}-\hat{f}_j(t)=\hat{f}_j(t)\left(\frac{1}{\alpha}-1\right)\\
&\leq (\lambda +1)\left(\frac{1}{\alpha}-1\right)< \delta,
\end{align*}
where the last inequality follows from our choice of $\alpha$.

We consider now the open sets $G_1=\{t\in G:\,d(t, C)<\frac{2\rho}{3}\}$ and $G_2=\{t\in G:\,\frac{\rho}{3}<d(t, C)\}$. We consider two functions $\varphi_i\prec G_i$, $i=1,2$ as in Lemma \ref{split} and for $i=1,2$, $j=1,2$ we define the function $\tilde{f}_j^i=\varphi_i \tilde{f}_j$.

Thus defined, the functions  $\tilde{f}_j^i$ satisfy

\begin{itemize}
\item Since $\tilde{f}_1(t)=\tilde{f}_2(t)$ for every $t\in G_1$, we get that  $\tilde{f}_1^1=\tilde{f}_2^1$.

\item $\supp(\tilde{f}_j^2)\subset G\setminus C$, for $j=1,2$.
\item $\tilde{f}_j=\tilde{f}_j^1\vee \tilde{f}_j^2$, for $j=1,2$.
\item $\supp(\tilde{f}_j^1\wedge\tilde{f}_j^2)\subset G\setminus C$.
\item For $i,j=1,2$, $\|\tilde{f}_j^i\|_\infty\leq \lambda+1$.
\end{itemize}

We note that, for  $j=1,2$, $$\tilde{V}(\tilde{f}_j)=\tilde{V}(\tilde{f}_j^1)+ \tilde{V}(\tilde{f}_j^2) - \tilde{V}(\tilde{f}_j^1\wedge \tilde{f}_j^2)).$$

Finally, using the fact that $\mu_{\lambda+1}(G\setminus C)<\epsilon$ and Observation \ref{Vpequenha}, we have
\begin{align*}
\left|\tilde{V}(f_1)-\tilde{V}(f_2)\right | &\leq \left|\tilde{V}(\hat{f}_1)-\tilde{V}(\hat{f}_2)\right |+2\epsilon \leq \left|\tilde{V}(\tilde{f}_1)-\tilde{V}(\tilde{f}_2)\right |+4\epsilon\\
&\leq \left| \tilde{V}(\tilde{f}_1^2) - \tilde{V}(\tilde{f}_1^1\wedge \tilde{f}_1^2)- \tilde{V}(\tilde{f}_2^2) +\tilde{V}(\tilde{f}_2^1\wedge \tilde{f}_2^2))\right|+4\epsilon\leq 8\epsilon.
\end{align*}

\end{proof}

\begin{lema}\label{l:aproxcontinuas}
Let $g=\sum_{i=1}^M a_i \chi_{A_i}$ be a positive simple function with $A_i\cap A_j=\emptyset$ for every $i\neq j$. Let $\lambda=\|g\|_\infty$. Then, for every $\rho>0, \epsilon>0$  there exists $f\in C(S^{n-1})^+$ and a set $A\subset S^{n-1}$ with $\mu_\lambda(S^{n-1}\setminus A)<\epsilon$ such that $|\tilde{V}(f)- \sum_{i=1}^M \nu_{a_i}(A_i)|<\rho$ and $f(t)=g(t)$ for every $t\in A$. Moreover, $f$ can be chosen so that $\|f\|_\infty\leq \lambda$.
\end{lema}

\begin{proof}
Without loss of generality we can assume that  $\bigcup_{i=1}^M A_i=S^{n-1}$.
Let $N=\sum_{k=2}^M {M \choose  k}$.
Using the  regularity of $\mu_\lambda$, for every $1\leq i \leq M$, we choose a closed set $K_i$ and an open set $G'_i$ such that  $K_i\subset A_i \subset G'_i$ and such that
$$
\mu_\lambda(G'_i\setminus K_i)<\min\{\frac{\epsilon}{M}, \frac{\rho}{21M}, \frac{\rho}{2N}\}.
$$

Next, for $1\leq i\leq M$, we define  $$G_i=G'_i\cap \bigcap_{j\not = i} K_j^c.$$ Note that we still have
\begin{equation}\label{eq:cond}
\mu_\lambda(G_i\setminus K_i)<\min\{\frac{\epsilon}{M}, \frac{\rho}{21M}, \frac{\rho}{2N}\}.
\end{equation}

Clearly,  $\bigcup_{i=1}^M G_i=S^{n-1}$. We use Lemma \ref{split} to choose a lattice partition of the unity $(\varphi_i)_{i=1}^M$ with $\varphi_i\prec G_i$ and $\bigvee_{i=1}^M \varphi_i=\uno$.

We define $f_i=a_i \varphi_i$ and $f=\bigvee_{i=1}^M f_i$. Then, on the one hand, for every $t\in A=\cup_{i=1}^M K_i$, $f(t)=g(t)$. Note that $$\mu_\lambda(S^{n-1}\setminus A)\leq \sum_{i=1}^M \mu_\lambda(G_i\setminus K_i)< \epsilon.$$

On the other hand, for every $1\leq i \leq M$, $K_i\prec\frac{f_i}{a_i}\prec G_i$. Therefore it follows from Lemma \ref{l:aproxnu} and condition (\ref{eq:cond}) that $$\left|\tilde{V}(f_i)-\nu_{a_i}(K_i)\right|<\frac{\rho}{3M}.$$ Also, for every $i\not = j$, $\supp(f_i\wedge f_j)\subset G_i\setminus K_i$, which, by Observation \ref{Vpequenha}, implies that, for $k\geq2$ and $1\leq i_1<i_2<\ldots<i_k\leq M$, we have that 
$$
\tilde{V}(f_{i_1}\wedge f_{i_2}\wedge\cdots \wedge f_{i_k})<\rho/2N.
$$

We can now apply \cite[Lemma 3.1]{Vi} and we get

\begin{align*}
\left|\tilde{V}(f)-\sum_{j=1}^M \nu_{a_j}(A_j)\right| & = \left|\sum_{k=1}^M(-1)^{k-1}\!\!\!\!\!\!\sum_{1\leq i_1<\ldots<i_k\leq M} \!\!\!\tilde{V}(f_{i_1}\wedge \cdots \wedge f_{i_k}) - \sum_{j=1}^M \nu_{a_j}(A_j)\right| \\
& \leq \left|\sum_{j=1}^M \tilde{V}(f_j) - \sum_{j=1}^M \nu_{a_j}(A_j)\right|+\sum_{k=2}^M\sum_{1\leq i_1<\ldots<i_k\leq M}\frac{\rho}{2N} \\
& \leq \sum_{j=1}^M \left|\tilde{V}(f_j) -  \nu_{a_j}(K_j)\right|+ \sum_{j=1}^M \nu_{a_j}(A_j\backslash K_j)+\frac{\rho}{2} \\
&<\rho.
\end{align*}

For the last part of the statement, note that, for $1\leq i \leq M$, $\|f_i\|_\infty\leq a_i$.
\end{proof}

Given a Borel set $A\subset S^{n-1}$ and $g\in B(\Sigma_n)$, we define $\|g\|_A=\sup_{t\in A} |g(t)|$.

The following lemma is a simple consequence of the fact that for $g\in C(S^{n-1})^+$, an any $A\subset S^{n-1}$, we have $\sup_{t\in A} |g(t)|=\sup_{t\in \overline{A}} |g(t)|$.

\begin{lema}\label{Cauchyadherencia}
Let $(f_i)_{i\in \mathbb N} \subset C(S^{n-1})^+$, and let $A\subset S^{n-1}$ be a Borel set. If the sequence of restrictions $({f_i}_{|_{A}})_{i\in\mathbb N}$ is a Cauchy sequence for the norm $\|\cdot\|_A$, then the sequence $({f_i}_{|_{\overline{A}}})_{i\in\mathbb N}$ (the sequence of restrictions to $\overline{A}$) is also a Cauchy sequence for the norm  $\|\cdot\|_{\overline{A}}$.
\end{lema}

We will need the following result of Dugundji which we state for completeness (\cite[Theorem 5.1]{Du}).
\begin{teo}\label{thm: Dugundji}
Let $K$ be a compact metric space, and let $A\subset K$ be a closed subset. Then there exists a norm one simultaneous extender, that is,  a norm one injective continuous linear mapping $T:C(A)\rightarrow C(K)$ such that, for every $f\in C(A)$, $T(f)_{|_A}=f$. Moreover, $T$ can be chosen so that, for every $f\in C(A)^+$, $T(f)\in C(K)^+$.
\end{teo}
\begin{proof}
Only the last  statement is not explicitly stated in \cite{Du}, but it follows immediately from the proof.
\end{proof}

We can now prove the following.
\begin{lema}\label{l:epsilonCauchy}
Let $\lambda\geq 0$, $\epsilon>0$. Let $B\subset S^{n-1}$ be a Borel set with $\mu_{\lambda+1}(B)<\epsilon$. Let $A=S^{n-1}\backslash B$ and let  $(f_i)_{i\in \mathbb N}\subset C(S^{n-1})^+$ be a sequence such that $({f_i}_{|_{A}})_{i\in\mathbb N}$ is a Cauchy sequence for the norm $\|\cdot\|_A$ and such that $\|f_i\|_\infty\leq \lambda$ for every $i\in \mathbb N$. Then, for every $\rho>0$ there exists $N\in \mathbb N$ such that, for every $p,q>N$, $$|\tilde{V}(f_p)-\tilde{V}(f_q)|\leq 16\epsilon+\rho.$$
\end{lema}

\begin{proof}
Using Lemma \ref{Cauchyadherencia} we may assume that $A$ is a closed set, thus $B$ is open. We consider the simultaneous extender $T:C(A)\rightarrow C(S^{n-1})$ of Theorem \ref{thm: Dugundji}. Then, for every $i\in \mathbb N$, $\|T(f_i)\|_\infty\leq \lambda$ and $\left(T(f_i)\right)_{i\in \mathbb N}\subset C(S^{n-1})^+$  is a Cauchy sequence for the supremum norm, hence converges to some $f\in C(S^{n-1})^+$. Therefore,  there exists $i_0$ such that, for every $p,q\geq i_0$,
$$
\left|\tilde{V}(T(f_p))-\tilde{V}(T(f_q))\right|<\rho.
$$

Lemma \ref{l:controldefunciones} implies that, for every $i\in \mathbb N$,
$$
\left|\tilde{V}(T(f_i))-\tilde{V}(f_i)\right|\leq 8\epsilon.
$$

Therefore, for every $i\geq i_0$, we have
\begin{align*}
\left|\tilde{V}(f_p)-\tilde{V}(f_q)\right|&\leq \left|\tilde{V}(f_p)-\tilde{V}(T(f_p)) \right|+ \left|\tilde{V}(T(f_p))- \tilde{V}(T(f_q))\right|+ \\ &+\left|\tilde{V}(T(f_q))-\tilde{V}(f_q)\right|<\rho+ 16\epsilon.
\end{align*}
\end{proof}

Finally, we can prove the result that will allow us to extend $\overline{V}$ to $B(\Sigma_n)^+$:

\begin{prop}\label{p:cauchysequences}
Let $(g_i)_{i\in \mathbb N}$ be a  Cauchy sequence of simple functions. Then $\left(\overline{V}(g_i)\right)_{i\in \mathbb N}$ is a Cauchy sequence of real numbers.
\end{prop}

\begin{proof}
Since $(g_i)_{i\in \mathbb N}$ is  Cauchy, there exists $\sup_{i\in\mathbb N}\|g_i\|_\infty=\lambda<\infty$. We fix $\epsilon>0$. According to Lemma \ref{l:aproxcontinuas}, for every $i\in \mathbb N$ there exist $f_i^\epsilon\in C(S^{n-1})^+$ such that
$|\tilde{V}(f_i^\epsilon)- \overline{V}(g_i)|<\epsilon$ and a Borel set $B_i^\epsilon\subset S^{n-1}$, with $\mu_{\lambda+1}(B_i^\epsilon)<\frac{\epsilon}{2^i}$, such that, for every $t\in A_i^\epsilon:= S^{n-1}\setminus B_i^\epsilon$,  $f_i^\epsilon(t)=g_i(t)$.

Let $B_\epsilon=\bigcup_{i\in \mathbb N}  B_i^\epsilon$. Then $\mu_{\lambda+1}(B_\epsilon)<\epsilon$, and let $A_\epsilon=S^{n-1}\setminus B_\epsilon$.

We can apply Lemma \ref{l:epsilonCauchy} and we get the existence of $N\in\mathbb N$ such that, for every $p,q\geq N$,  $$\left|\tilde{V}(f_p^\epsilon)-\tilde{V}(f_q^\epsilon)\right|<17\epsilon.$$

Therefore, for $p,q\geq N$ we have
$$\left|\overline{V}(g_p)-\overline{V}(g_q)\right|\leq \left|\overline{V}(g_p)-\tilde{V}(f_p^\epsilon)\right|+ \left|\tilde{V}(f_p^\epsilon)-\tilde{V}(f_q^\epsilon)\right|+ \left|\tilde{V}(f_q^\epsilon)-\overline{V}(g_q)\right|\leq 19\epsilon. $$
\end{proof}

Therefore, $\overline{V}:S(\Sigma_n)^+\longrightarrow \mathbb R^+$ can be extended uniquely to a continuous function, which we will denote equally $\overline{V}:B(\Sigma_n)^+\longrightarrow \mathbb R^+$. Namely, given $f\in B(\Sigma_n)^+$ and \emph{any} sequence $(f_n)\subset S(\Sigma_n)^+$ such that $\|f_n-f\|_\infty\rightarrow0$ we can set 
\begin{equation}\label{eq:defextension}
\overline{V}(f)=\lim_n \overline{V}(f_n).
\end{equation}
By Proposition \ref{p:cauchysequences}, the limit above always exists and does not depend on the choice of $(f_n)\subset S(\Sigma_n)^+$.

Moreover, note that given $f,g\in B(\Sigma_n)^+$ and $(f_n),(g_n)\subset S(\Sigma_n)^+$ such that $\|f_n-f\|_\infty\rightarrow 0$ and $\|g_n-g\|_\infty\rightarrow0$ it follows that $f_n\vee g_n,f_n\wedge g_n\in S(\Sigma_n)^+$, $\|f_n\vee g_n-f\vee g\|_\infty\rightarrow 0$ and $\|f_n\wedge g_n-f\wedge g\|_\infty\rightarrow0$, thus we have
\begin{align*}
\overline V(f\vee g)+\overline V(f\wedge g)&=\lim_n \overline V(f_n\vee g_n)+\lim_n \overline V(f_n\wedge g_n)\\
&=\lim_n\overline V(f_n\vee g_n)+\overline V(f_n\wedge g_n)\\
&=\lim_n \overline V(f_n)+\overline V(g_n)\\
&=\lim_n \overline V(f_n)+\lim_n\overline V(g_n)\\
&=\overline V(f)+\overline V(g).
\end{align*}
This means that $\overline V$ is a continuous valuation on $B(\Sigma_n)^+$.

We show next that $\overline{V}$ is actually an extension of $\tilde{V}$.

\begin{prop}\label{p:extension} Let $\tilde V:C(S^{n-1})^+\longrightarrow \mathbb R^+$ and $\overline{V}:B(\Sigma_n)^+\longrightarrow \mathbb R^+$ be as above. Then, for every $f\in C(S^{n-1})^+$, $\tilde{V}(f)=\overline{V}(f)$.
\end{prop}

\begin{proof}
Let $f\in C(S^{n-1})^+$. We will construct two sequences $(g_j)_{j\in \mathbb N}\subset S(\Sigma_n)^+$, $(f_j)_{j\in \mathbb N}\subset C(S^{n-1})^+$ such that, for every $j\in \mathbb N$,
\begin{eqnarray}
\|g_j -f\|_\infty\leq\frac{1}{j},\label{eq:g_jf}\\
\|f_j-f\|_\infty\leq \frac{2}{j},\label{eq:f_jf}\\
|\overline{V}(g_j)-\tilde{V}(f_j)|\leq  \frac{1}{j}.\label{eq:g_jf_j}
\end{eqnarray}

The proof will be  finished once we have constructed such sequences  $(g_j)_{j\in \mathbb N}$, $(f_j)_{j\in \mathbb N}$,  since, in that case,  $$\overline{V}(f)=\lim_j \overline{V}(g_j)=\lim_j \tilde{V}(f_j)=\tilde{V}(f).$$

We proceed to the construction of the sequences $(g_j)_{j\in \mathbb N}$, $(f_j)_{j\in \mathbb N}$. Let $\lambda=\|f\|_\infty$.

For each $j\in \mathbb N$ we make the following construction:
Let $\delta=\frac{1}{j}$.  We define $M=\left[\frac{\|f\|_\infty}{\delta}\right] + 1$, where $[x]$ denotes the integer part of $x$.
Let $A_1=f^{-1}([0,\delta])$ and, for $2\leq i \leq M$,
$$
A_i=f^{-1}\left(((i-1)\delta,i\delta]\right).
$$
Now, we define $g_j=\sum_{i=1}^M i\delta \chi_{A_i}$. Clearly, \eqref{eq:g_jf} follows, since
\begin{equation*}
\|g_j-f\|_\infty\leq \delta=\frac{1}{j}.
\end{equation*}

For $1\leq i \leq M$, we proceed as in Lemma \ref{l:aproxcontinuas} to choose a closed set $K'_i$ and an open set $G'_i$ such that $K'_i\subset A_i\subset G'_i$, and $\mu_{\lambda+1}(G'_i\setminus K'_i)<\frac{\delta}{14M}.$

Next, define $K_1=A_1$, and, for $2\leq i \leq M$,
$$
K_i=K'_i\cup f^{-1}\left(\left[\Big(i-\frac{99}{100}\Big)\delta,i\delta\right]\right).
$$

Now, for $1\leq i \leq M$, define $$G_i^{''}=G'_i\cap \bigcap_{k\not = i} K_k^c.$$

Finally, define $$G_1=G^{''}_1\cap f^{-1}\left(\left[0,\Big(i+\frac{1}{100}\Big)\delta\right)\right)$$ and, for $2\leq i \leq M$,  $$G_i=G^{''}_i\cap f^{-1}\left(\left((i-1)\delta,\Big(i+ \frac{1}{100}\Big)\delta\right)\right).$$

Then we have that:
\begin{itemize}
\item $K_i\subset A_i\subset G_i$ for $1\leq i\leq M$,
\item $\mu_{\lambda+1}(G_i\setminus K_i)<\frac{\delta}{14M}$ for $1\leq i\leq M$,
\item $K_i\cap G_{i'}=\emptyset $ if $i\not = i'$,
\item $\bigcup_{i=1}^M G_i=S^{n-1}$, and
\item $G_i\cap G_{i'}=\emptyset$ if $|i-i'|>1$.
\end{itemize}

We apply again  Lemma \ref{split} to choose a lattice partition of the unity $(\varphi_i)_{i=1}^M$ with $\varphi_i\prec G_i$ and $\bigvee_{i=1}^M \varphi_i=\uno$. Then, we  define $h_i=i\delta \varphi_i$ and set
$$
f_j=\bigvee_{i=1}^M h_i.
$$
Note that for every $t\in K_i$, since $K_i\cap G_{i'}=\emptyset $ for $i\not = i'$, we have
$$
f_j(t)=h_i(t)=i\delta=g_j(t).
$$
Otherwise, for $t\in G_i\setminus K_i$, since $G_i\cap G_{i'}=\emptyset$ if $|i-i'|>1$, there are only two possibilities: $t\in G_i\cap G_{i-1}$ or $t\in G_i\cap G_{i+1}$. In the former case we have $f_j(t),g_j(t)\in[(i-1)\delta,i\delta]$, while in the latter we have $f_j(t),g_j(t)\in[i\delta,(i+1)\delta]$.
Therefore,
$$
\|f_j-g_j\|_\infty\leq\delta.
$$
From this together with \eqref{eq:g_jf}, we get \eqref{eq:f_jf}:
$$
\|f_j-f\|_\infty\leq 2\delta.
$$
The coincidence of $h_i$ and $g_j$ on $K_i$, together with the fact that $\mu_{i\delta}(G_i\backslash K_i)<\frac{\delta}{14M}$, imply, by Lemma \ref{l:aproxnu} that
$$
\left|\tilde{V}(h_i)-\nu_{i\delta}(K_i)\right|<\frac{\delta}{2M}.
$$

Moreover, note that if $i'\not \in \{i-1, i, i+1\}$, then $(h_i\wedge h_{i'})=0$. Otherwise, if $i'\in \{i-1, i+1\}$, $\supp(h_i\wedge h_{i'}) \subset G_i\setminus K_i$. Also, for every three different indexes $i,i',i^{''}$, we have $h_i\wedge h_{i'}\wedge h_{i^{'''}}=0$. Therefore, applying \cite[Lemma 3.1]{Vi} again we get
\begin{align*}
\left|\tilde V(f_j)-\overline V(g_j)\right|&=\left|\tilde{V}\Big(\bigvee_{i=1}^M h_i\Big)-\sum_{i=1}^M \nu_{i\delta}(A_i)\right| \\
&\leq \left| \sum_{ i=1}^ M \tilde{V}(h_i) - \sum_{i=1}^M \nu_{i\delta}(A_i)\right| + \left|\sum_{i=1}^{M-1} \tilde{V}(h_{i}\wedge h_{i+1})\right| \\
&\leq�\frac{\delta}{2}+\frac{\delta}{14}<\delta=\frac{1}{j}.
\end{align*}
This proves \eqref{eq:g_jf_j} and the result follows.
\end{proof}

This finishes the proof of Theorem \ref{t:extension} and, hence, also the proof of Theorem \ref{t:extensiontoborel}. 

Now we can prove  Theorem \ref{t:integral}. The precise statement is

\begin{teo}\label{t:representation}
Let $\tilde V:C(S^{n-1})^+\rightarrow\mathbb R$ be a radial continuous valuation. If we consider its extension $\overline V:B(\Sigma_n)^+\rightarrow \mathbb R$ given by Theorem \ref{t:extensiontoborel}, then there exists a measure $\mu$ defined on the Borel $\sigma$-algebra of $S^{n-1}$ and a function $K:\mathbb R^+\times S^{n-1}\rightarrow \mathbb R$ such that, for every $g\in S(\Sigma_n)^+$, we have
$$
\overline V(g)=\int_{S^{n-1}} K(g(t),t)d\mu(t).
$$
\end{teo}

\begin{proof}
We first consider the radial continuous valuation $\tilde{V}'(f)=\tilde{V}(f)-\tilde{V}(0)$ together with its extension $\overline{V}'$. Using Theorem \ref{t:jordan} we can write $\tilde{V}'=\tilde{V}_1-\tilde{V}_2$, both of them positive valuations with $\tilde{V}_i(0)=0$.

For $i=1,2$ and $\lambda\geq 0$, we consider the corresponding representing and control measures $\nu^i_\lambda$ and $\mu^i_\lambda$ as in Section \ref{controlmeasure}.

For every $\lambda\geq 0$, we define the measure $\mu_\lambda=\mu_\lambda^1 + \mu_\lambda^2$, and we also define the  normalized control measure $\mu$ by 
$$
\mu=\sum_{k=1}^\infty \frac{\mu_{k}}{2^k\mu_{k}(S^{n-1})}.
$$

It is clear from the definitions that, for every $\lambda\geq 0$, for $i=1,2$,  the measure $\nu^i_\lambda$ is continuous with respect to $\mu^i_\lambda$. Since the family of control measures $\mu_\lambda$ are clearly monotonous with respect to $\lambda$, it follows that, for each $\lambda\geq 0$, $\mu_\lambda$ is continuous with respect to $\mu$ and, hence, also $\nu_\lambda:=\nu_\lambda^1-\nu_\lambda^2$ is continuous with respect to $\mu$. By Radon-Nikodym's theorem, for every $\lambda\geq 0$ there exist a function $K'_\lambda\in L^1(\mu)$ such that
$$
\nu_\lambda(A)=\int_A K'_\lambda(t)d\mu(t),
$$
for every $A\in\Sigma_n$. Let $K':\mathbb R^+\times S^{n-1}\rightarrow \mathbb R$ be the function given by 
$$
K'(\lambda,t)=K'_\lambda(t).
$$
Using the fact that $K'(0,t)=0$ $\mu$-a.e. $t$, for every $A\in\Sigma_n$ we have
$$
\nu_\lambda(A)=\int_{S^{n-1}} K'(\lambda\chi_A(t),t)d\mu(t).
$$
Therefore, for $g=\sum_{j=1}^n a_j\chi_{A_j}\in S(\Sigma_n)^+$ with pairwise disjoint $(A_j)_{j=1}^n$, we have
$$
\overline{ V}'(g)=\sum_{j=1}^n \nu_{a_j}(A_j)=\sum_{j=1}^n \int_{A_j}K'(a_j,t)d\mu(t)=\int_{S^{n-1}} K'(g(t),t)d\mu(t).
$$

Defining $K(\lambda, t)=K'(\lambda, t)+ \tilde{V}(0)$, we finish the proof. 
\end{proof}

\section{Previous work and open questions}\label{s:openquestions}

Integral representations in the spirit of Riesz theorem have been previously considered for certain classes of (not necessarily linear) functionals on spaces $C(K)$. In particular, in a series of papers \cite{CF,FK,FK:69}, N. Friedman et al. studied integral representations for additive functionals in spaces $C(K)$. Let us briefly recall their main result and terminology:

\begin{defi}
Given a compact Hausdorff space $K$, a functional $\phi:C(K)\rightarrow \mathbb R$ is called:
\begin{enumerate}
\item \emph{Additive}, if for any $f_1,f_2,f\in C(K)$ with $|f_1|\wedge|f_2|=0$, it follows that
$$
\phi(f_1+f_2+f)=\phi(f_1+f)+\phi(f_2+f)-\phi(f).
$$
\item \emph{Bounded on bounded sets}, if for each $m>0$, there is $M(m)>0$ such that $|\phi(f)|\leq M(m)$ whenever $\|f\|_\infty\leq m$.
\item \emph{Uniformly continuous on bounded sets}, if for every $\epsilon>0$ and $m>0$, there is $\delta(\epsilon,m)>0$ such that $|\phi(f)-\phi(g)|\leq \epsilon$ whenever $\|f-g\|_\infty<\delta(\epsilon,m)$ with $\|f\|_\infty,\|g\|_\infty\leq m$. 
\end{enumerate}
\end{defi}

\begin{teo}\label{t:Friedman}
Given a compact Hausdorff space $K$, and a functional $\phi:C(K)\rightarrow \mathbb R$, the following are equivalent:
\begin{enumerate}
\item $\phi$ is additive, bounded on bounded sets and uniformly continuous on bounded sets. 
\item There exist a measure $\mu$ of finite variation defined on the Borel $\sigma$-algebra of $K$, and a function $f:\mathbb R\times K\rightarrow \mathbb R$ such that
\begin{enumerate}
\item $f(x,\cdot)$ is measurable for every $x$,
\item $f(\cdot,t)$ is continuous for $\mu$-almost every $t$,
\item for each $m>0$ there is $C_m>0$ such that $|f(x,t)|\leq C_m$ for $\mu$-almost every $t$, whenever $|x|\leq m$,
\end{enumerate}
such that for every $g\in C(K)$
$$
\phi(g)=\int_K f(g(t),t)d\mu(t).
$$
\end{enumerate}
\end{teo}

We will see now that additivity is the same as the property defining a valuation:

\begin{lema}\label{l:additivefunctional}
A mapping $\phi:C(K)_+\rightarrow \mathbb R$ is an additive functional if and only if for every $f,g\in C(K)_+$
$$
\phi(f)+\phi(g)=\phi(f\vee g)+\phi(f\wedge g).
$$
\end{lema}

\begin{proof}
Suppose first $\phi:C(K)_+\rightarrow \mathbb R$ is an additive functional, that is
$$
\phi(f_1+f_2+f)=\phi(f_1+f)+\phi(f_2+f)-\phi(f)
$$
whenever $f_1\wedge f_2=0$. Given $f,g\in C(K)_+$, let $f_1=f-f\wedge g$ and $f_2=g-f\wedge g$. It is clear that $f_1\wedge f_2=0$, hence
\begin{align*}
\phi(f\vee g)&=\phi(f+g-f\wedge g)=\phi(f_1+f_2+f\wedge g)\\
&=\phi(f_1+f\wedge g)+\phi(f_2+f\wedge g)-\phi(f\wedge g)\\
&=\phi(f)+\phi(g)-\phi(f\wedge g).
\end{align*}
Therefore, $\phi$ is a valuation.

Conversely, let us suppose that $\phi:C(K)_+\rightarrow \mathbb R$ is a valuation and take $f_1,f_2,f\in C(K)_+$ with $f_1\wedge f_2=0$. We have that
\begin{align*}
\phi(f_1+f)+\phi(f_2+f)&=\phi((f_1+f)\vee(f_2+f))+\phi((f_1+f)\wedge(f_2+f))\\
&=\phi((f_1\vee f_2)+f)+\phi((f_1\wedge f_2)+f)\\
&=\phi(f_1+f_2+f)+\phi(f),
\end{align*}
which yields that $\phi$ is an additive functional.
\end{proof}

As a side remark, note that  every valuation clearly defines an {\em orthogonally additive functional}, that is $\phi(f+g)=\phi(f)+\phi(g)$ whenever $f\wedge g=0$. However, not every orthogonally additive functional is a valuation, as the following simple example shows:

\begin{example}
Let $\phi:C(K)_+\rightarrow \mathbb R$ be given by $\phi(f)=\min\{f(t):t\in K\}.$ If $f\wedge g=0$, then we have that $\phi(f)=\phi(g)=\phi(f\wedge g)=0$, so trivially $\phi$ is orthogonally additive. However, we can consider a partition of $K$ into two sets $A,B$ with $A\cap B=\emptyset$ and functions $f_A,g_B\in C(K)_+$ such that $f_A(t)=1$ for every $t\in A$, $g_B(t)=1$ for every $t\in B$ and for some $t_A\in A$ and $t_B\in B$ we have $f_A(t_B)=0$ and $g_B(t_A)=0$. It follows that
$$
\phi(f_A)=\phi(g_B)=\phi(f_A\wedge g_B)=0,
$$
while $\phi(f_A\vee g_B)=1$. Therefore, $\phi$ cannot be a valuation. Note that $\phi$ is  continuous and satisfies $\phi(0)=0$.
\end{example}

As we mentioned above the functionals under consideration in the work of Friedman et al. satisfy the additional assumptions of being bounded and uniformly continuous on bounded sets. Note that by Lemmas \ref{l:bbs} and \ref{l:additivefunctional}, being bounded on bounded sets follows from continuity. We do not know whether continuity for valuations is actually enough to obtain uniform continuity on bounded sets. Note this last hypothesis is heavily used in \cite{CF,FK,FK:69} to obtain the desired integral representation.

Our main open questions now are the following

\begin{question}\label{q:1} Is every radial continuous valuation on the star bodies of $S^{n-1}$ uniformly continuous on bounded sets?
\end{question}

\begin{question}\label{q:2} Is the integral representation in Theorem \ref{t:integral} valid for every star body?
\end{question}

If Question \ref{q:1} would be true, then Theorem \ref{t:Friedman} would imply a positive answer to Question \ref{q:2}. So far, we do not even know that the function $K(\lambda, t)$ is measurable in the first variable for $\mu$ almost every $t$. Uniform continuity in bounded sets would imply that $K(\lambda, t)$ would be continuous $\mu$-almost everywhere in the first variable \cite{CF}, and the validity of the integral representation for continuous functions would follow.

The ``2 dimensional'' densities $K(\lambda,t)$ appearing in the integral representation of Theorem \ref{t:representation} have certain continuity in the first variable, which is however not yet sufficient to answer the questions above. 

\begin{lema}\label{l:uniformeenconjuntos}
Given $\lambda\geq 0$,  for every $\epsilon>0$ there exists $\delta>0$ such that for every Borel set $A\subset S^{n-1}$, if  $|\lambda-\lambda'|<\delta$ then $$|\nu_\lambda(A)-\nu_{\lambda'}(A)|<\epsilon.$$
\end{lema}

\begin{proof}
The continuity of $\overline{V}$ implies that, given $\epsilon$, there exists $\delta$ such that, for every $g\in B(\Sigma_n)^+$, if $\|\lambda\uno -g\|_\infty<\delta$ then $|\overline{V}(\lambda\uno)-\overline{V}(g)|<\delta$.

Let $A\subset S^{n-1}$ be a Borel set. Using that $\nu_\lambda(S^{n-1})=\overline{V}(\lambda\uno)$, and defining $g=\lambda' \chi_A + \lambda\chi_{A^c}$, we get  $$|\nu_\lambda(A)-\nu_{\lambda'}(A)|=|\nu_\lambda(A)+\nu_{\lambda}(A^c)-\nu_{\lambda'}(A)-\nu_\lambda(A^c)|=|\overline{V}(\lambda\uno)-\overline{V}(g)|<\epsilon.$$
\end{proof}

\begin{prop}\label{p:convergenciaennorma}
Let $\lambda\geq 0$. For every $\epsilon>0$ there exists $\delta>0$ such that, if $|\lambda-\lambda'|<\delta$ then $$\|K_\lambda-K_{\lambda'}\|_{L_1(\mu)}<2\epsilon. $$
\end{prop}

\begin{proof}
Given $\varphi \in L_1(\mu)$, if we define $A=\varphi^{-1}([0,\infty))$, then $$\|\varphi\|_{L_1(\mu)}= \int_{S^{n-1}} \varphi(t)(\chi_A(t)-\chi_{A^c}(t) )d\mu(t).$$

Given $\epsilon>0$, we take $\delta$ as in Lemma \ref{l:uniformeenconjuntos} and considering $\varphi=K_\lambda-K_{\lambda'}$,   we have
\begin{align*}
\|K_\lambda-K_{\lambda'}\|_{L_1(\mu)}&=\int_{S^{n-1}} \varphi(t)(\chi_A(t)-\chi_{A^c}(t) )d\mu(t)\\
&=|\nu_\lambda(A)-\nu_{\lambda}(A^c)- \left(  \nu_{\lambda'}(A)-\nu_{\lambda'}(A^c)\right)|   \\
&\leq |\nu_\lambda(A)-\nu_{\lambda'}(A)| + |\nu_{\lambda}(A^c)  -\nu_{\lambda'}(A^c)|<2\epsilon. 
\end{align*}
\end{proof}

Finally, the next fact provides some additional information related to Question \ref{q:1}.

\begin{prop}\label{p:unifcontwk}
Let $V:C(S^{n-1})^+\longrightarrow \mathbb R$ be a radial continuous valuation. Then, for every weakly compact subset $W\subset  C(S^{n-1})^+$, $V$ is uniformly continuous on $W$.
\end{prop}

\begin{proof}
Suppose the contrary. That is, there exist a weakly compact set $W\subset  C(S^{n-1})^+$, $\epsilon>0$ and two sequences $(f_n)_{n\in\mathbb N}, (g_n)_{n\in\mathbb N}\subset W$ such that 
\begin{equation}\label{eq:fn-gnvaa0}
\|f_n-g_n\|_\infty\rightarrow 0,
\end{equation}
while for every $n\in\mathbb N$
\begin{equation}\label{eq:Vfn-Vgnseparado}
|V(f_n)-V(g_n)|\geq\epsilon.
\end{equation}

Taking into account that $W$ is weakly compact, by the Eberlein-Smulian Theorem (cf. \cite{Diestel}), passing to a further subsequence we can assume that $f_n\rightarrow f$ in the weak topology, for certain $f\in C(S^{n-1})^+$. 

Since the point evaluations are continuous linear functionals in $C(S^{n-1})$, we thus have that $f_n\rightarrow f$ pointwise. Let $\lambda=\sup_{h\in W}\|h\|_\infty$. Now, by Egoroff's Theorem (cf. \cite{Dudley}), there is a Borel set $A\subset S^{n-1}$ with $\mu_\lambda(S^{n-1}\setminus A)<\epsilon/17$ such that $\|f_n-f\|_A=\sup_{t\in A}|f_n(t)-f(t)|\rightarrow 0$. By \eqref{eq:fn-gnvaa0}, we also have $\|g_n-f\|_A\rightarrow0$. Therefore, Lemma \ref{l:epsilonCauchy} yields in particular that for some $N\in\mathbb N$ and every $n\geq N$
$$
|V(f_n)-V(g_n)|<\epsilon,
$$
which is a contradiction with \eqref{eq:Vfn-Vgnseparado}.
\end{proof}

In connection with Question \ref{q:1}, if $V:C(S^{n-1})^+\longrightarrow \mathbb R$ is a radial continuous valuation which is not uniformly continuous on bounded sets, then there must be some bounded sequence $(f_n)_{n\in\mathbb N}\subset C(S^{n-1})^+$ and perturbations $(\tilde f_n)_{n\in\mathbb N}$ with $\|f_n-\tilde f_n\|_\infty\rightarrow0$ but $|V(f_n)-V(\tilde f_n)|\geq \epsilon$ for some $\epsilon>0$.  Proposition \ref{p:unifcontwk} yields that no subsequence of $(f_n)_{n\in\mathbb N}$ can be weakly Cauchy, hence, by Rosenthal's $\ell_1$ Theorem (cf. \cite[Chapter XI]{Diestel}), the sequence $(f_n)_{n\in\mathbb N}$ must be equivalent to the unit basis of $\ell_1$ in the sense that $\|\sum_{n} a_nf_n\|\approx\sum_n |a_n|$ (and should be a Rademacher-like sequence, see \cite[Chapter XI]{Diestel} for details).


\begin{thebibliography}{99}

\bibitem{CF} R. V. Chacon, N. Friedman, Additive functionals. {\em Arch. Rational Mech. Anal.} {\bf 18} (1965), 230--240. 

\bibitem{Diestel} J. Diestel, Sequences and series in Banach spaces. Graduate Texts in Mathematics, 92. Springer-Verlag, New York, 1984.

\bibitem{Dudley} R. M. Dudley, Real analysis and probability. Revised reprint of the 1989 original. Cambridge Studies in Advanced Mathematics, 74. Cambridge University Press, Cambridge, 2002.

\bibitem{Du} J. Dugundji, An extension of Tietze's Theorem, {\em Pacific J. of Math.} {\bf 1} (1951), 353--367.

\bibitem{DGP} P. Dulio, R. J. Gardner, and C. Peri, Characterizing the dual mixed volume via additive functionals, {\em Indiana Univ. Math. J.} {\bf 65} (2016), 69--91.

\bibitem{E}  R. Engelking, {\em General topology}. Sigma Series in Pure Mathematics, 6. Heldermann Verlag, Berlin, 1989.

\bibitem{FK} N. A. Friedman, M. Katz, A representation theorem for additive functionals. {\em Arch. Rational Mech. Anal.} {\bf 21} (1966), 49--57.

\bibitem{FK:69}N. A. Friedman, M. Katz, On additive functionals. {\em Proc. Amer. Math. Soc.} {\bf 21} (1969) 557--561.

\bibitem{Ga1} R. J. Gardner, A positive answer to the Busemann-Petty problem in three dimensions, {\em  Ann. of Math.} (2) {\bf 140} (1994), 435--447.

\bibitem{Gabook} R. J. Gardner, Geometric Tomography, second edition, Encyclopedia of Mathematics and its Applications, vol. 58, Cambridge University Press, Cambridge, 2006.

\bibitem{Ga2} R. J. Gardner, A. Koldobsky, T. Schlumprecht, An analytical solution to the Busemann--Petty problem on sections of convex bodies, {\em Ann. of Math.} (2) {\bf 149} (1999) 691--703.

\bibitem{Halmos} P. Halmos, {\em Measure Theory}, Springer-Verlag, New York, 1974.

\bibitem{Klain96} D. A. Klain, Star Valuations and Dual Mixed Volumes, {\em Adv. Math.} {\bf  121} (1996), 80--101.

\bibitem{Klain97} D. A. Klain, Invariant  Valuations on Star-Shaped Sets, {\em Adv. Math.} {\bf  125} (1997), 95--113.

\bibitem{Lu1} M. Ludwig, Intersection bodies and valuations, {\em American Journal of Mathematics}, {\bf 128} (2006),  1409--1428.

\bibitem{Lu2} M. Ludwig, M. Reitzner, A classification of SL(n) invariant valuations, {\em Ann. of Math.} {\bf 172} (2010), 1219--1267.

\bibitem{Lut_mv1} E. Lutwak, Dual Mixed volumes, {\it Pacific J. Math} {\bf 58} (1975), 531--538.

\bibitem{Vi} I. Villanueva, Radial continuous rotation invariant valuations on star bodies. {\em Adv. Math.} {\bf 291} (2016), 961--981.

\bibitem{Zh} G. Zhang, A positive answer to the Busemann-Petty problem in four dimensions, {\em Ann. of Math.} (2) {\bf 149} (1999) 535--543.

\end{thebibliography}
\end{document}